\documentclass[a4paper,10pt]{amsart}

\usepackage{graphicx}
\usepackage{dirtytalk}
\usepackage{amsthm}
\usepackage{amssymb}
\usepackage{amsmath}
\usepackage{accents}
\usepackage{hyperref}
\usepackage{xcolor}
\usepackage{enumerate}
\usepackage{listings}
\usepackage{complexity}
\usepackage{blkarray}
\usepackage{braket}
\usepackage{lmodern}
\usepackage[english]{babel}
\usepackage{enumitem}

\usepackage[font=small,labelfont=bf]{caption}
\usepackage[font=small,labelfont=normalfont,labelformat=simple]{subcaption}
\graphicspath{{figs/}}

\newtheorem{theorem}{Theorem}[section]

\newtheorem{lemma}[theorem]{Lemma}

\newtheorem{observation}[theorem]{Observation}
\newtheorem{proposition}[theorem]{Proposition}
\newtheorem{conjecture}[theorem]{Conjecture}
\newtheorem{problem}[theorem]{Problem}

\author
{
Raphael Steiner 
}
\thanks{Department of Computer Science, Institute of Theoretical Computer Science, ETH Z\"{u}rich, Switzerland,  \texttt{raphaelmario.steiner@inf.ethz.ch}. The research of the author is funded by the Ambizione Grant No. 216071 of the Swiss National Science Foundation.}

\date{\today}

\title{On the difference between the chromatic and cochromatic number}

\begin{document}
\maketitle

\begin{abstract}
The \emph{cochromatic number} $\zeta(G)$ of a graph $G$ is the smallest number of colors in a vertex-coloring of $G$ such that every color class forms an independent set or a clique. In three papers written around 1990, Erd\H{o}s, Gimbel and collaborators raised several open problems regarding the relationship of the chromatic and cochromatic number of a graph. In this short note, we address several of these problems, in particular 
\begin{itemize}
    \item we disprove a conjecture of Erd\H{o}s, Gimbel and Straight from 1988,
    \item answer negatively a problem posed by Erd\H{o}s and Gimbel in 1993, and
    \item give positive evidence for a 1000\$--question of Erd\H{o}s and Gimbel. 
\end{itemize}
\end{abstract}

\section{Introduction}
 While the \emph{chromatic number} $\chi(G)$ of a graph $G$ is famously defined as the minimum size of a partition of the vertex-set $V(G)$ into independent sets, the \emph{cochromatic number} $\zeta(G)$ of a graph $G$, introduced by Lesniak and Straight in 1977~\cite{lesniakstraight1977}, is defined as the minimum size of a partition of $V(G)$ into sets that are \emph{homogeneous} (i.e., independent or a clique). 
By definition, we  clearly have $\zeta(G)\le \chi(G)$ for every graph $G$. While complete graphs are examples showing that in general no inverse relationship holds, in a sequence of papers published around 1990~\cite{erdosgimbel1993,erdosgimbelstraight1990chromatic, erdosgimbelkratsch91}, Erd\H{o}s, Gimbel and collaborators studied the difference between those two notions of coloring from various perspectives, and established several non-trivial bounds on the difference $\chi(G)-\zeta(G)$. Along the way, they posed a number of interesting conjectures and problems, some of which we address in this note.

\medskip

\paragraph{\textbf{Graphs of bounded clique number.}} In 1988~\cite{erdosgimbelstraight1990chromatic}, Erd\H{o}s, Gimbel and Straight proved that the difference $\chi(G)-\zeta(G)$ can be upper-bounded in terms of the clique number $\omega(G)$ of the graph. Concretely, for every integer $n>2$ they defined $f(n)$ as the smallest integer such that every graph $G$ other than $K_{n-1}$ with $\omega(G)<n$ satisfies $\chi(G)\le \zeta(G)+f(n)$. They proved that $f(n)$ is well-defined and grows exponentially with $n$. They also gave independent proofs that $f(3)=0$ and $f(4)=1$, but the precise value of $f(n)$ for $n\ge 5$ remains unknown. As a first contribution, we observe the (conceptually interesting) fact that for every fixed $n\ge 5$, determining the precise value $f(n)$ reduces to a finite computation.

\begin{proposition}\label{prop:main2}
Let $n\ge 5$ and $f\ge n-2$ be integers. Then there exists an explicit constant $N=N(n,f)\in \mathbb{N}$ such that $f(n)\le f$ if and only if $\chi(G)\le \zeta(G)+f$ for all graphs $G$ with $\omega(G)<n$ of order at most $N$.
\end{proposition}

To see that Proposition~\ref{prop:main2} indeed shows that for every fixed integer $n\ge 5$, determining the value of $f(n)$ boils down to a finite computation, note that Erd\H{o}s, Gimbel and Straight showed in~\cite{erdosgimbelstraight1990chromatic} that $f(n)\ge 3(n-3)/2$ if $n$ is odd and $f(n)\ge (3n-10)/2$ if $n$ is even, implying $f(n)\ge n-2$ for $n\ge 5$. Hence, for every of the finitely many values of $f$ that lie between $n-2$ and the upper bound on $f(n)$ given in~\cite{erdosgimbelstraight1990chromatic}, one can use Proposition~\ref{prop:main2} to determine in finite time whether $f(n)\le f$, and hence find the value of $f(n)$. At least in the case of $f(5)$, this may indeed be a practically feasible approach. The short proof of Proposition~\ref{prop:main2} is presented in Section~\ref{sec:proof}. 

Regarding bounds on the value of $f(5)$, Erd\H{o}s, Gimbel and Straight found several examples of $K_5$-free graphs $G\not\simeq K_4$ such that $\chi(G)>\zeta(G)+2$, showing that $f(5)\ge 3$. However, all their examples have $\zeta(G)\le 3$ and they noted that they could not find any examples of such graphs with $\zeta(G)>3$. This led them to pose the following conjecture, which, if true, would be best-possible. 

\begin{conjecture}[Erd\H{o}s, Gimbel, Straight 1988~\cite{erdosgimbelstraight1990chromatic}]\label{conj:main}
Every graph $G$ with $\omega(G)<5$ and $\zeta(G)>3$ satisfies $\chi(G)\le \zeta(G)+2$. 
\end{conjecture}

This $36$-year old conjecture has been repeated several times, such as by Erd\H{o}s and Gimbel~\cite{erdosgimbel1993} in 1993, by Jensen and Toft in their book on Graph Coloring Problems (see~\cite{jensentoft}, Problem~17.3), and it also appears on the Erd\H{o}s problems website maintained by Thomas Bloom as \href{https://www.erdosproblems.com/762}{Problem No.~762}. In 2006, Liu, Chen and Ou~\cite{xinsheng2006lower} proved the conjecture for a class of line graphs.

A closely related open problem was posed by Erd\H{o}s and Gimbel in 1993~\cite{erdosgimbel1993}, who state that they constructed several hundred examples of graphs $G$ with $\omega(G)<5$ and $\chi(G)>\zeta(G)+2$ but suspect that there may be only finitely many such $G$.
\begin{problem}[Erd\H{o}s and Gimbel 1993~\cite{erdosgimbel1993}]\label{prob:finite}
Are there finitely many graphs $G$ with $\omega(G)<5$ and $\chi(G)>\zeta(G)+2$? 
\end{problem}

As the main result of this paper, we disprove Conjecture~\ref{conj:main} and negatively answer Problem~\ref{prob:finite} via the following result.

\begin{theorem}\label{thm:main}
There are infinitely many graphs $G$ with $\omega(G)<5$, $\zeta(G)=4$ and $\chi(G)=7$. 
\end{theorem}
The proof of Theorem~\ref{thm:main} is given in Section~\ref{sec:cex}. An interesting question that is left open is whether the conclusion of Conjecture~\ref{conj:main} remains false if we require $\zeta(G)$ to be very large.
\begin{problem}
    Let $k\ge 5$. Is there a graph $G_k$ with $\omega(G_k)<5$, $\zeta(G_k)= k$ and $\chi(G_k)=\zeta(G_k)+3$?
\end{problem}

\medskip

\paragraph*{\textbf{Random graphs.}} Moving away from graphs of bounded clique number and considering random graphs instead, Erd\H{o}s and Gimbel raised the following open problem in their paper~\cite{erdosgimbel1993} from 1993.

\begin{problem}[Erd\H{o}s and Gimbel 1993~\cite{erdosgimbel1993}]\label{prob:erdosgimbel}
For every $n\in \mathbb{N}$, let $G_n\sim G(n,\frac{1}{2})$ be an Erd\H{o}s-R\'{e}nyi random graph on $n$ vertices. Is it true that, almost surely, $$\chi(G_n)-\zeta(G_n)\rightarrow \infty$$ for $n\rightarrow \infty$?
\end{problem}

Problem~\ref{prob:erdosgimbel} is listed on the Erd\H{o}s problems page as \href{https://www.erdosproblems.com/625}{Problem No.~625}, where it is mentioned that Erd\H{o}s offered a monetary prize of 100\$ for a positive and an even higher prize of 1000\$ for a negative resolution of Problem~\ref{prob:erdosgimbel}. It is also mentioned (with more background and related problems) in a recent article by Gimbel, see Problem~2 in~\cite{gimbel16}. It is known, see~\cite{erdosgimbel1993,gimbel16} that w.h.p., we have
$$\frac{n}{2\log_2(n)}\le \zeta(G_n)\le \chi(G_n)\le (1+o(1))\frac{n}{2\log_2(n)},$$ so in particular we have that $\chi(G_n)/\zeta(G_n)\rightarrow 1$ almost surely.

As the last new contribution of this paper, we present some modest positive evidence towards Problem~\ref{prob:erdosgimbel}, showing that for infinitely many values of $n$ the difference $\chi(G_n)-\zeta(G_n)$ becomes large with probability bounded away from zero, so in particular, $\chi(G_n)-\zeta(G_n)$ is large in expectation.

\begin{theorem}\label{thm:main2}
For every $\varepsilon>0$ there exists an absolute constant $c>0$ such that for infinitely many values $n\in \mathbb{N}$ we have 
$$\mathbb{P}(\chi(G_n)-\zeta(G_n)\ge n^{1/2-\varepsilon})\ge c$$ and thus $\mathbb{E}(\chi(G_n)-\zeta(G_n))\ge \Omega(n^{1/2-\varepsilon})$.
\end{theorem}

Theorem~\ref{thm:main2} is deduced from the recent breakthroughs by Heckel~\cite{heckel2021}, and Heckel and Riordan~\cite{heckel2023} on the non--concentration of the chromatic number via the Harris-FKG-inequality, the short proof is given in Section~\ref{sec:proofmain2}. It would be interesting to see whether the tools from~\cite{heckel2021,heckel2023} can be adapted to the cochromatic setting to make further progress on Problem~\ref{prob:erdosgimbel}. 

\medskip

\paragraph{\textbf{Notation and Terminology.}} For a graph $G$, we denote by $V(G)$ its vertex set and by $E(G)$ its edge set.
A subset of vertices $X$ of a graph $G$ is said to be \emph{independent} if the induced subgraph $G[X]$ has no edges, and a \emph{clique} if $G[X]$ is a complete graph. We say that $X$ is \emph{homogeneous} if $X$ is an independent set or a clique. By $\omega(G)$ we denote the \emph{clique number}, i.e. the size of a largest clique. A \emph{proper $k$-coloring} of a graph $G$ is a mapping $c:V(G)\rightarrow \{1,\ldots,k\}$ such that $c(x)\neq c(y)$ for every $xy\in E(G)$.

\section{Proof of Proposition~\ref{prop:main2}}\label{sec:proof}
In this section, we present the short proof of Proposition~\ref{prop:main2}. We start with a simple observation. By $R(n,n)$ we denote the $n$-th diagonal Ramsey number.
\begin{observation}\label{obs:ramsey}
Let $n>2$, and let $g(n)$ denote the maximum chromatic number of a graph on less than $R(n,n)$ vertices with clique number less than $n$. 

Then every graph $G$ with $\omega(G)<n$ and $|V(G)|\ge R(n,n)$ satisfies
$$\chi(G)\le \left\lceil\frac{|V(G)|-R(n,n)+1}{n}\right\rceil+g(n).$$
\end{observation}
\begin{proof}
As long as $|V(G)|\ge R(n,n)$, take an independent set of size $n$ (which has to exist since $\omega(G)<n$) and remove it from $G$. Repeating this process until we obtain a graph on less than $R(n,n)$ vertices yields the desired result. 
\end{proof}
We can now prove Proposition~\ref{prop:main2}.
\begin{proof}[Proof of Proposition~\ref{prop:main2}]
Define $$N=N(n,f):=\max\left\{n(n-1)\left(g(n)-\frac{R(n,n)-1}{n}-f\right), R(n,n)-1\right\}.$$ (We believe this value is far from optimal and did not put effort into minimizing it.)

By definition, we have that $f(n)\le f$ is equivalent to the statement (A) ``Every graph $G$ with $\omega(G)<n$ other than $K_{n-1}$ satisfies $\chi(G)\le \zeta(G)+f$.''
Since we assumed $f\ge n-2$, the graph $K_{n-1}$ satisfies $\chi(K_{n-1})=n-1\le 1+f=\zeta(K_{n-1})+f$, and hence statement (A) is equivalent to the statement (B) ``Every graph $G$ with $\omega(G)<n$ satisfies $\chi(G)\le \zeta(G)+f$''.

To show the proposition, it suffices to prove that every counterexample $G$ to statement (B) of minimum order $|V(G)|$ (if it exists), satisfies $|V(G)|\le N$. 

So, towards a contradiction suppose that $G$ is a smallest counterexample with $|V(G)|>N$. Let $k:=\zeta(G)$ and let  $X_1,\ldots,X_k$ be a partition into homogeneous sets. We claim that $X_i$ is a clique for every $i=1,\ldots,n$. Indeed, suppose towards a contradiction that $X_i$ is independent for some~$i$. Then $\zeta(G-X_i)=k-1$ and, since we assumed $G$ to be smallest counterexample, we have $\chi(G-X_i)\le \zeta(G-X_i)+f$. We then obtain $\chi(G)\le \chi(G-X_i)+1\le \zeta(G-X_i)+f+1=\zeta(G)+f$, a contradiction. Next, note that since $N\ge R(n,n)-1$, by Observation~\ref{obs:ramsey} we have 
$$\chi(G)< \frac{|V(G)|}{n}+g(n)-\frac{R(n,n)-1}{n}+1\le \frac{(n-1)}{n}k+g(n)-\frac{R(n,n)-1}{n}+1,$$
where we used that $|V(G)|=\sum_{i=1}^{k}|X_i|\le k(n-1)$ since $\omega(G)<n$ and every $X_i$ is a clique. But we also have $\chi(G)\ge\zeta(G)+f+1=k+f+1$, so we find (after rearranging):
$$k<n\left(g(n)-\frac{R(n,n)-1}{n}-f\right).$$
Hence, we find
$N<|V(G)|\le k(n-1)<n(n-1)\left(g(n)-\frac{R(n,n)-1}{n}-f\right)\le N,$ the desired contradiction. This concludes the proof of the proposition.
\end{proof}

\section{Counterexample to the Erd\H{o}s-Gimbel-Straight conjecture}\label{sec:cex}

In this section, we prove Theorem~\ref{thm:main}. We start by constructing an auxiliary graph $H$ with special properties, as follows.

\begin{lemma}\label{lem:aux}
There exists a graph $H$ on $11$ vertices with the following properties.

\begin{enumerate}
    \item $\omega(H)<5$,
    \item $V(H)$ can be partitioned into $3$ cliques,
    \item For every proper $6$-coloring $c:V(H)\rightarrow \{1,\ldots,6\}$, there exists a set $X\subseteq V(H)$ such that $\omega(H[X])<4$ and $c(X)=\{1,\ldots,6\}$.
\end{enumerate}
\end{lemma}
Before giving the proof of the lemma, we record the following useful observation.

\begin{observation}\label{obs:obvious}
For every proper $3$-coloring $c$ of $C_5$ and for every $i \in \{1,2,3\}$, there exist two non-adjacent vertices $a,b$ on $C_5$ such that $\{c(a),c(b)\}=\{1,2,3\}\setminus \{i\}$.
\end{observation}
\begin{proof}
Let the vertices along the cycle be enumerated as $\{v_1,v_2,v_3,v_4,v_5\}$. Then (up to symmetry and renaming of colors), we may assume w.l.o.g. $c(v_1)=1$, $c(v_2)=2$, $c(v_3)=3$, $c(v_4)=2$, $c(v_5)=3$. It is now easily observed that the claimed statement is true. 
\end{proof}

We can now prove Lemma~\ref{lem:aux}.
\begin{proof}[Proof of Lemma~\ref{lem:aux}]
Let us define $H$ as follows: We start from the disjoint union of two cycles $C_1, C_2$, each of length $5$, by adding all the edges $\{uv|u \in V(C_1), v \in V(C_2)\}$. Then, we pick two vertices $x_1\in V(C_1), x_2 \in V(C_2)$ and add a new vertex $v\notin V(C_1)\cup V(C_2)$ that is made adjacent to only $x_1$ and $x_2$.

Let us verify that this graph $H$ on $11$ vertices has the desired properties: 

Since $v$ is of degree $2$, it cannot be included in any $K_5$, and since $C_1$ and $C_2$ are triangle-free graphs, the maximum size of a clique in $H-v$ is $4$. Hence, $\omega(H)<5$. This establishes property~(1). 

To verify~(2), note that $\{v,x_1,x_2\}$ is a clique in $H$ and that $H-\{v,x_1,x_2\}$ can be partitioned into two cliques of size four (pair up the edges of perfect matchings of $C_1-x_1$ and $C_2-x_2$, respectively). 

Finally, let us verify~(3). Let $c:V(H)\rightarrow \{1,\ldots,6\}$ be a proper coloring. Then, since $\chi(C_1)=\chi(C_2)=3$ and $C_1$ and $C_2$ are fully connected to each other, they must each contain exactly $3$ of the $6$ colors. W.l.o.g. (possibly after permuting colors and due to symmetry) we may assume $c(V(C_1))=\{1,2,3\}$, $c(V(C_2))=\{4,5,6\}$ and $c(v) \in \{1,2,3\}$. By Observation~\ref{obs:obvious}, there exist non-adjacent vertices $a, b \in V(C_1)$ such that $\{c(a),c(b),c(v)\}=\{1,2,3\}$. Furthermore, we can pick $3$ vertices $u_4, u_5, u_6\in V(C_2)$ with $c(u_i)=i$ for $i \in \{4,5,6\}$. Let now $X=\{a,b,v,u_4,u_5,u_6\}$. Then clearly $c(X)=\{1,2,3,4,5,6\}$. Furthermore, we claim that $\chi(G[X])\le 3$. This can be seen as follows: Trivially, the graph $H[\{u_4,u_5,u_6\}]$ is a forest, and since $v$ is connected to at most one of these three vertices, also $H[\{v,u_4,u_5,u_6\}]$ is a forest, and thus $\chi(H[\{v,u_4,u_5,u_6\}])\le 2$. Since $\{a,b\}$ is an independent set, it follows that $\chi(H[X])\le 1+2=3$, and thus in particular $\omega(H[X])\le 3$. This verifies the statement~(3) of the lemma and concludes its proof.

\end{proof}

With Lemma~\ref{lem:aux} at hand, we can now easily complete the proof of Theorem~\ref{thm:main}. 

\begin{proof}[Proof of Theorem~\ref{thm:main}]
Let $H$ be the graph from Lemma~\ref{lem:aux} and denote by $\mathcal{X}$ the collection of all subsets $X\subseteq V(H)$ such that $\omega(H[X])<4$. Let $G$ be any graph obtained from $H$ by adding, for every $X\in \mathcal{X}$, a non-empty independent set $V_X$ of vertices, such that each vertex in $V_X$ has neighborhood $X$, and such that the sets $(V_X|X\in \mathcal{X})$ are pairwise disjoint and each disjoint from $V(H)$. Note that there are infinitely many graphs $G$ that can be built in this way.

We claim that $G$ satisfies the properties required by the theorem, i.e. $\omega(G)<5$, $\zeta(G)=4$, and $\chi(G)=7$. 

The first property is easily seen: Every clique in $G$ is either a clique in $H$ and thus of size at most $4$ since $\omega(H)<5$ by item~(1) of Lemma~\ref{lem:aux}, or it consists of one vertex $v\in V_X$ for some $X\in \mathcal{X}$ together with a clique in $H[X]$, so it is of size at most $1+\omega(H[X])<1+4=5$.

For the second property, first note that $\zeta(G)\le 4$, since by item~(2) of Lemma~\ref{lem:aux}, we can partition $V(H)$ into $3$ cliques, and since the newly added vertices $\bigcup_{X\in \mathcal{X}}V_X$ form an independent set. On the other hand, it is impossible to partition $V(G)$ into $3$ homogeneous sets. To see this, first note that since $V(H)$ can be partitioned into $3$ cliques, every independent set in $G$ contains at most $3$ vertices from $V(H)$, and since $\omega(G)<5$, every clique has size at most $4$. 
Therefore, if it were possible to cover $V(G)$ by three homogeneous sets, then since $|V(H)|=11$ and $|V(G)|>12$, exactly one of the three homogeneous sets would need to be independent, while the other two would have to be cliques of size exactly $4$ that are contained in $V(H)$. In particular, the independent set would need to contain $3$ vertices from $H$ and all of  $\bigcup_{X\in \mathcal{X}}V_X$. It is easy to see however that such a set cannot be independent, for instance since every vertex $w \in V(H)$ is connected to each vertex in $V_{\{w\}}$. Thus, we indeed have $\zeta(G)=4$, as desired.

Finally, let us show that $\chi(G)=7$. It is easy to see that $\chi(G)\le \chi(H)+1=6+1=7$. Towards a contradiction, suppose that $G$ admits a proper $6$-coloring $c:V(G)\rightarrow \{1,\ldots,6\}$. Then by item~(3) of Lemma~\ref{lem:aux}, there exists a set $X\in \mathcal{X}$ such that $c(X)=\{1,2,3,4,5,6\}$. However, now there is no available color for the vertices in $V_X$ in the proper coloring, a contradiction.

This shows that $G$ satisfies all required properties and concludes the proof.
\end{proof}
\section{Proof of Theorem~\ref{thm:main2}}\label{sec:proofmain2}
In this section, we present the short proof of Theorem~\ref{thm:main2}. We use the following anti-concentration result for the chromatic number proved by Heckel and Riordan~\cite{heckel2023}.
\begin{theorem}[cf. Theorem~5 in~\cite{heckel2023}]\label{thm:heckel}
Fix $p\in (0,1)$ and $\varepsilon>0$, and let $([s_n,t_n])_{n \in \mathbb{N}}$ be a (deterministic) sequence of intervals such that $\mathbb{P}(\chi(G_{n,p})\in [s_n,t_n])\rightarrow 1$ as $n\rightarrow \infty$. Then there are infinitely many $n$ such that $t_n-s_n>n^{\frac{1}{2}-\varepsilon}$. 
\end{theorem}
We also need the following special case of the Harris-FKG correlation inequality.
\begin{theorem}[Harris-FKG inequality, cf. Chapter 6 of~\cite{alon2015}]\label{thm:harris}
Let $\Omega=\{0,1\}^m$ for some $m\in \mathbb{N}$ be a product space and $A, B \subseteq \Omega$ events such that both $A$ and $B$ are increasing, that is, for any two $\mathbf{x},\mathbf{y}\in \Omega$ we have that $\mathbf{x} \in A$ and $\mathbf{x} \le \mathbf{y}$ (coordinate-wise) implies $\mathbf{y}\in A$, and the same holds for $B$. Then 
$$\mathbb{P}(A\cap B)\ge \mathbb{P}(A)\mathbb{P}(B).$$
\end{theorem}

We are now ready for the proof of Theorem~\ref{thm:main2}.
\begin{proof}[Proof of Theorem~\ref{thm:main2}]
In the following, we denote by $G_n$ the Erd\H{o}s-Renyi random graph on $n$ vertices with edge-probability $p=\frac{1}{2}$.
We start by observing that there exists some constant $\delta \in (0,1)$ such that the following statement holds for infinitely many values of $n\in \mathbb{N}$:

\medskip

($\ast$) For every $s\in \mathbb{R}$ we have $\mathbb{P}(s\le \chi(G_n)\le s+n^{\frac{1}{2}-\varepsilon})\le 1-\delta$.

\medskip

Indeed, suppose towards a contradiction that for every $\delta>0$ the above statement is only valid for finitely many $n$.

This means that for every $\delta>0$ there exists some $n_0=n_0(\delta)\in \mathbb{N}$ and for every $n\ge n_0$ a real interval $[s_n^\delta,t_n^\delta]$ of length exactly $n^{\frac{1}{2}-\varepsilon}$ such that $$\mathbb{P}(\chi(G_n)\in [s_n^\delta,t_n^\delta])> 1-\delta.$$
W.l.o.g. assume $n_0(\delta)\rightarrow \infty$ for $\delta\rightarrow 0$.

Now, for every $n\in \mathbb{N}$ set $\delta(n):=\min\{\frac{1}{k}|n\ge n_0(\frac{1}{k}), k \in \mathbb{N}\}$. We clearly have $\delta(n)\rightarrow 0$ for $n\rightarrow \infty$ and
$$\mathbb{P}(\chi(G_n)\in [s_n^{\delta(n)},t_n^{\delta(n)}])\ge 1-\delta(n)\rightarrow 1$$ for $n\rightarrow \infty$. However, there exists no $n$ such that $t_n^{\delta(n)}-s_n^{\delta(n)}>n^{\frac{1}{2}-\varepsilon}$, a contradiction to Theorem~\ref{thm:heckel}.

Hence, we know that ($\ast$) holds for infinitely many values of $n$. Now, consider any $n\in \mathbb{N}$ for which $(\ast)$ holds and let $s^\ast:=\min\{s\in \mathbb{N}_0|\mathbb{P}(\chi(G_n)\le s)\ge \delta/2\}$. We then have
$$\mathbb{P}(\chi(G_n)\le s^\ast+n^{1/2-\varepsilon})\le \mathbb{P}(\chi(G_n)\le s^\ast-1)+\mathbb{P}(s^\ast\le \chi(G_n)\le s^\ast+n^{1/2-\varepsilon})$$
$$<\delta/2+(1-\delta)=1-\delta/2.$$
Finally, this implies that $\mathbb{P}(\chi(G_n)>s^\ast+n^{1/2-\varepsilon})\ge\delta/2$. 

In the following, denote by $\overline{G_n}$ the complement graph of $G_n$ and note that it follows the same distribution as $G_n$. Now let $A$ and $B$ be the probability events corresponding to the statements $\chi(G_n)\ge s^\ast+n^{1/2-\varepsilon}$ and $\chi(\overline{G_n})\le s^\ast$, respectively. We can consider them as events in a product probability space $\Omega=\{0,1\}^{\binom{n}{2}}$, where every coordinate corresponds to a pair of vertices and we associate an existing edge in $G_n$ with a $1$-entry at the corresponding position. We can now observe that $A$ and $B$ are both increasing events, i.e., they continue to hold when more edges are added to an instance of $G_n$. Hence, we may apply the Harris-FKG inequality (Theorem~\ref{thm:harris}) to find

$$\mathbb{P}(\chi(G_n)\ge s^\ast+n^{1/2-\varepsilon} \text{ and }\chi(\overline{G_n})\le s^\ast)=\mathbb{P}(A\cap B)\ge\mathbb{P}(A)\mathbb{P}(B)$$
$$=\mathbb{P}(\chi(G_n)\ge s^\ast+n^{1/2-\varepsilon})\mathbb{P}(\chi(\overline{G_n})\le s^\ast)$$
$$=\mathbb{P}(\chi(G_n)\ge s^\ast+n^{1/2-\varepsilon})\mathbb{P}(\chi(G_n)\le s^\ast)\ge (\delta/2)^2.$$
Finally, using the inequality $\zeta(G)\le \chi(\overline{G})$, which holds for every graph $G$, we find that 
$\mathbb{P}(\chi(G_n)-\zeta(G_n)\ge n^{1/2-\varepsilon})\ge (\delta/2)^2=:c>0$, and this concludes the proof of the theorem.
\end{proof}
\bibliographystyle{abbrvurl}
\bibliography{references}

\end{document}